\documentclass[11pt]{amsart}

\usepackage[marginratio=1:1]{geometry}
\usepackage{color}
\definecolor{refs}{rgb}{0.7,0,0}
\definecolor{ext}{RGB}{112,112,112}
\definecolor{cite}{RGB}{034,113,179}
\usepackage[colorlinks=true,citecolor=cite,linkcolor=refs,urlcolor=ext,backref=page]{hyperref}
\usepackage{xifthen}

\newtheorem{theorem}{Theorem}[section]
\newtheorem{lemma}[theorem]{Lemma}

\newtheorem{proposition}[theorem]{Proposition}

\theoremstyle{definition}
\newtheorem*{definition}{Definition}
\newtheorem*{remark}{Remark}
\newtheorem*{example}{Example}
\newtheorem*{examples}{Examples}

\newcommand{\R}{\mathbb{R}}
\newcommand{\N}{\mathbb{N}}
\newcommand{\C}{\mathbb{C}}

\renewcommand{\i}{\mathrm{i}}

\newcommand{\V}[1][3]{{\mathcal{H}_{#1}}}
\newcommand{\spn}{\operatorname{span}}
\newcommand{\tr}{\operatorname{tr}}
\newcommand{\fJ}{\mathfrak{J}}
\newcommand{\fj}{\mathfrak{j}}
\newcommand{\fg}{\mathfrak{g}}
\newcommand{\gl}{\mathfrak{gl}}

\author[W.~Kry\'nski]{Wojciech Kry\'nski}

\title[GL(2)-geometry and complex structures]{GL(2)-geometry and complex structures}

\address{
Institute of Mathematics, Polish Academy of Sciences, ul. \'Sniadeckich 8, 00-656 Warszawa, Poland}

\email{krynski@impan.pl}

\begin{document}
\maketitle

\begin{abstract}
We study $GL(2)$-structures on differential manifolds. The structures play a fundamental role in the geometric theory of ordinary differential equations. We prove that any $GL(2)$-structure on an even dimensional manifold give rise to a certain almost-complex structure on a bundle over the original manifold. Further, we exploit a natural notion of integrability for the $GL(2)$-structures, which is a counterpart of the self-duality for the 4-dimensional conformal structures. We relate the integrability of the $GL(2)$-structures to the integrability of the almost-complex structures. This allows to perform a twistor-like construction for the $GL(2)$-geometry.  Moreover, we provide an explicit construction of a canonical connection for any $GL(2)$-structure.
\end{abstract}

\section{Introduction}
A $GL(2)$-structure on a manifold $M$ is a smooth field of rational normal curves in the tangent bundle $TM$. The curves reduce the full frame bundle of $M$ to a $GL(2)$-subbundle, where $GL(2)$ acts irreducibly on each fiber. In this way the $GL(2)$-structures can be studied in the classical framework of $G$-structures. Among the $GL(2)$-structures one can distinguish a natural subclass of \emph{integrable} structures, see \cite{FK,K1,KM} where they were referred to as involutive or $\alpha$-integrable structures. Roughly, we shall say that a structure is $i$-integrable if $M$ admits a family of $i$-dimensional submanifolds that are tangent to the iterated tangential varieties of the aforementioned rational curves in $TM$. The existence of the submanifolds allows to consider the integrable $GL(2)$-structures as natural counterparts of the Einstein--Weyl geometry in dimension 3 or anti-self-dual metrics in dimension 4 (as follows from the classical results of E.~Cartan and R.~Penrose \cite{C,P}). 

In the recent years, the $GL(2)$-structures, also reffered to as the paraconformal structures, attracted much attention due to their significance in the geometric theory of ordinary differential equations (ODEs), see \cite{B,DT,GN,K1}. In particular the 4-dimensional structures introduced by R.~Bryant in the seminal paper \cite{B} appeared in the context of the exotic holonomy groups. All $GL(2)$-structures arising from ODEs are integrable. In dimension 4 the notion of integrability coincides with the torsion-freeness considered in \cite{B}. On the other hand, in higher dimensions the torsion necessairly appears (we refer to \cite{GN,AS} for a detailed study of 5-dimensional structures).

In \cite{FK,KM} it was shown that the $GL(2)$-structures can arise also as characteristic varieties of certain systems of partial differential equations (PDEs). The integrability of the $GL(2)$-structures reflects the integrability of PDEs. This phenomenon has been observed before in \cite{FK0} in the context of 3-dimensional conformal metrics of Lorentzian signature (see also \cite{CK}). The 3-dimentional metrics are exactly the 3-dimensional $GL(2)$-structures (see \cite{C,DK,FKN,N0} for the geometry of these structures in relation to ODEs).

The purpose of the present paper is to investigate the $GL(2)$-geometry through the naturally associated complex geometry. We shall mostly consider the $GL(2)$-structures on manifolds of even dimension $n=2l$ that are $l$-integrable. The structures will be also referred to as the \emph{half-integrable} $GL(2)$-structures. In the next section, Section \ref{secPre}, we shall provide a formal definition of the integrability. This will be followed by some examples and basic properties of the structures. Then, in Section \ref{secLinear}, we deal with algebraic properties of the standard irreducible representation of $GL(2)$ on the space $\V[k]$ of $k$-th order homogeneous polynomials in 2 variables. In particular we introduce a canonical complex structure on $\V[k]$. In Sections \ref{secGL2} and \ref{secCon} we utilize the algebraic  construction of Section \ref{secLinear} in order to introduce a canonical almost-complex structure for any even dimensional $GL(2)$-structure (in odd dimensions one can define a CR-structure). 

Our construction of the almost-complex structure is divided into two steps: Firstly, we define an almost-complex structure for a $GL(2)$-geometry equipped with a $GL(2)$-connection. This is a content of Section \ref{secGL2}. Secondly, in Section~\ref{secCon}, we provide an explicit construction of a canonical connection for any $GL(2)$-structure, which is a result of an independent interest.  The canonical connection gives a canonical almost-complex structure for even-dimensional $GL(2)$-structures. In dimension 4 our canonical connection coincide with the Bryant connection introduced in \cite{B}.

Finally, in Section \ref{secHol} we study the integrability of the almost-complex structures in terms of the integrability of the original $GL(2)$-structure. We also provide some further results on the holomorphic sections of the complex manifold coming from a half-integrable $GL(2)$-structure in a spirit of the twistor theory, see \cite{LM,M2}.

\paragraph*{\bf Acknowledges.} I am indebted to Thomas Mettler for many inspiring conversations that had a significant impact on the paper.

\section{Preliminaries on $GL(2)$-structures}\label{secPre}
We start with a formal definition of a $GL(2)$-structure, as introduced in \cite{DT}. For this, let us fix a manifold $M$ of dimension $k+1$. A $GL(2)$-structure on $M$ is an isomorphism of vector bundles
\begin{equation}\label{isoGL2}
TM\simeq S^k(E)
\end{equation}
where $E\to M$ is a rank-2 vector bundle over $M$ and $S^k(E)$ denotes $k$-th symmetric tensor product of $E$.

Let us fix a basis $(x,y)$ of $E_p$, where $p\in M$. This establishes an isomorphism of $S^k(E_p)$ and the space $\V[k]$ of homogeneous polynomials of order $k$ in two variables $x$ and $y$. Moreover, the natural right action of $GL(2)$ on $E_p$ with the chosen basis give rise to the following right action of $GL(2)$ on $S^k(E_p)$ 
\begin{equation}\label{eqaction}
(V\cdot g)(x,y)=V((x,y)g),
\end{equation}
where $V=V(x,y)\in\V[k]$ is a polynomial and $g\in GL(2)$. This yields a reduction of the full frame bundle $F\to M$ to a $GL(2)$-bundle $B(E)\to M$ whose fibers consist of frames of the form
$$
\chi=(x^k,x^{k-1}y,\ldots,xy^{k-1},y^k).
$$
$GL(2)$ acts on $B(E)$ via the standard irreducible representation $GL(2)\to GL(\V[k])$ defined by \eqref{eqaction}. The corresponding  irreducible representation of the Lie algebra $\gl(2)$  on $\V[k]$ will be denoted $\fg_k$, i.e. $\fg_k$ is the subset of $\V[k]^*\otimes\V[k]$ consisting of
$$
\left(\begin{array}{ccccc}
k\phi^1_1 & \phi^2_1 & 0 & \cdots &0\\
k\phi^1_2 & (k-1)\phi^1_1+\phi^2_2 & 2\phi^2_1 & \cdots& 0 \\ 
0 & (k-1)\phi^1_2 & (k-2)\phi^1_1+2\phi^2_2 & \cdots & 0 \\ 
\vdots & \vdots & \vdots & \ddots & \vdots \\
0 & 0 & 0 &\cdots & k\phi^2_2
\end{array}\right),
$$
where $\left(\begin{array}{cc}
\phi^1_1 & \phi^2_1\\
\phi^1_2 & \phi^2_2 
\end{array}\right)\in\gl(2)$.

Conversely, any reduction of the frame bundle $F\to M$ to a $GL(2)$-bundle, where $GL(2)$ acts irreducibly on each fiber of $F$, comes from an identification $TM\simeq S^k(E)$, since $\V[k]$ is the unique $(k+1)$-dimensional irreducible representation of $GL(2)$.

A $GL(2)$-structure on $M$ defines the following cone in each tangent space $T_pM$
$$
C_p(E)=\left\{\underbrace{e\odot\cdots\odot e}_k\ |\ e\in E_p\right\},
$$
where $\odot$ stands for a symmetric tensor product in $E_p\otimes E_p$, so $\underbrace{e\odot\cdots\odot e}_k$ is an element of $S^k(E_p)$.
We shall say that $C_p(E)$ is a cone of null directions of  a $GL(2)$-structure. If a basis $(x,y)$ of $E_p$ is chosen then, under the isomorphism $T_pM\simeq\V[k]$, the null directions correspond to polynomials of the form $V(x,y)=(ax+by)^k$. Passing to the projective space $P(T_pM)$ the cone $C_p(E)$ becomes a rational normal curve of order $k$ (also called the Veronese curve or twisted cubic in dimension 4).

\begin{example}
In dimension 3 the identification $T_pM\simeq \V[2]$ assigns to a vector $V\in T_pM$ a polynomial $V(x,y)=ax^2+bxy+cy^2$, and the coefficients $(a,b,c)$ became linear coordinates on $T_pM$. The null cone consists of $V(x,y)$ with multiple root, i.e. it is given by the discriminant equation $\Delta=b^2-4ac=0$. Note that $\Delta$ in this framework is a quadratic form on $T_pM$ and the corresponding bi-linear form defines a scalar product on $T_pM$ of Lorentzian signature. The cone of null directions of the metric coincides with $C_p(E)$.
\end{example}

Clearly $C_p(E)$ is $GL(2)$-invariant. Moreover, it is easy to observe that any field of cones $p\mapsto C_p$, where each $C_p$ is a rational normal curve in $T_pM$, defines the corresponding $GL(2)$-structure uniquely.

The osculating cone (or the tangential variety) of $C_p(E)$ is the collection of all tangent vectors to $C_p(E)$ considered as a subset of $T_pM$. Proceeding by induction we define $(i+1)$-th osculating cone of $C_p(E)$ as the osculating cone of $i$-th osculating cone of $C_p(E)$, where $0$-th osculating cone is $C_p(E)$ itself. Note that the $i$-th osculating cone of $C_p(E)$ is a $(2+i)$-dimensional subvariety of $T_pM\simeq\V[k]$ consisting of homogeneous polynomials with multiple roots of order $k-i$.

\begin{definition}
Let $TM\simeq S^k(E)$ be a $GL(2)$-structure on $M$. Pick $p\in M$ and $V\in C_p(E)$. An $\alpha_i$-\emph{plane} at $V$ is the tangent space at $V$ to the $(i-2)$-th osculating cone of $C_p(E)$.
\end{definition}

It follows from the definition that any $\alpha_i$-plane is an $i$-dimensional subspace of $T_pM$ for some $p\in M$. Moreover, if a basis $(x,y)$ of $E$ is chosen then any $\alpha_i$-plane is spanned by
$$
(ax+by)^k,(ax+by)^{k-1}(cx+dy),\ldots,(ax+by)^{k-i+1}(cx+dy)^{i-1},
$$
for some constants $a,b,c,d\in\R$ such that $ad-bc\neq 0$. Now we are in a position to introduce the notion of integrability for $GL(2)$-structures.

\begin{definition}
Let $TM\simeq S^k(E)$ be a $GL(2)$-structure on $M$. A submanifold $N\subset M$ of dimension $i$ is called $\alpha_i$-\emph{submanifold} of $M$ if all tangent spaces $T_pN$,  $p\in N$, are $\alpha_i$-planes.
\end{definition}
\begin{definition}
A $GL(2)$-structure on manifold $M$ is $i$-\emph{integrable} if there is a family of $\alpha_i$-submanifolds in $M$ such that any $\alpha_i$-plane at any point $p\in M$ is tangent to an $\alpha_i$-submanifold in the family.
If $M$ is a manifold of even dimension $n=2l$ then an $l$-integrable $GL(2)$-structure on $M$ is called a \emph{half-integrable} $GL(2)$-structure.
\end{definition}
\begin{example}
A particularly interesting family of integrable $GL(2)$-structures comes from so-called Veronese webs. A Veronese web is a special 1-parameter family of corank-1 foliations $\mathcal{F}=\{\mathcal{F}_t\}_{t\in\R}$, introduced by Glefand and Zakharevich in connection to bi-Hamiltonian systems \cite{GZ}. Namely, a Veronese web is a family of foliations such that in a neighborhood of any point there exist point-wise independent 1-forms $\omega_0,\ldots,\omega_k$ such that
$$
T\mathcal{F}_t=\ker\omega(t),
$$
where  $\omega(t)=\omega_0+t\omega_1+\cdots+t^k\omega_k$. Then, $t\mapsto\R\omega(t)$ is a field of rational normal curves in $P(T^*M)$. The dual curves in $P(TM)$ are defined as
$$
t\mapsto\ker\{\omega(t), \omega'(t),\ldots,\omega^{(k-1)}(t)\}.
$$
Due to the fact that all $\omega_i$ are independent the dual curves are well defined rational normal curves in $P(TM)$. Thus they define a $GL(2)$-structure on $M$. The integrability condition (for the family of foliations)
$$
d\omega(t)\wedge\omega(t)=0
$$
is satisfied for any $t$ and it translates to the fact that that  the corresponding $GL(2)$-structure is $k$-integrable. In fact it is $i$-integrable for all $i=1,\ldots,k$ acroding to Statement \eqref{st1} of Proposition \ref{propInt} below. It is proved in \cite{K1} that Veronese webs are described by the following hierarchy of integrable systems (in dimension 3 this is the Hirota equation related to the Einstein--Weyl structures \cite{DK}; see also \cite{KP} for other PDEs that descent to 3-dimensional Veronese webs)
$$
(a_i-a_j)\partial_0w\partial_i\partial_jw+a_j\partial_iw\partial_j\partial_0w-a_i\partial_jw\partial_i\partial_0w=0, \qquad i,j=1,\ldots,k,
$$
where $a_i$'s are distinct constants. Note that the characteristic variety of this system coincides with the field of the rational normal curves in the cotangent bundle that define the structure (see \cite{K1} for details and \cite{FK} for generalizations to arbitrary $k$-integrable $GL(2)$-structures).
\end{example}
\vskip 2ex
We have the following
\begin{proposition}\label{propInt}
Let $M$ be a manifold of dimension $k$.
\begin{enumerate}
\item\label{st1} If a $GL(2)$-structure on $M$ is $i$-integrable, then it is $j$-integrable for all $j<i$.
\item\label{st2} A $GL(2)$-structure on $M$ is $k-1$-integrable if and only if it is $k$-integrable.
\end{enumerate}
\end{proposition}
\begin{proof}
See the proof of \cite[Theorem 3.2]{K1}.
\end{proof}
The notion of $k$-integrability for $(k+1)$-dimensional $GL(2)$-structures was studied in \cite{K1} and \cite{FK} (under the name $\alpha$-integrability). It is immanently related to the geometry of ODEs, see e.g.~\cite{B,DT,K1}. In the present paper shall we concentrate on half-integrable structures provided that $k+1$ is even. In dimension $4$ the two notions, i.e. 2-integrability (or half-integrability in this case) and 3-integrability, coincide due to Statement \eqref{st2} of Proposition \ref{propInt}. In this case the family of 2-dimensional $\alpha$-manifolds is uniquely determined by the structure and is related to contact geometry of ODEs. On the other hand, the family of 3-dimensional $\alpha$-manifolds is ambiguous and is related to a choice of a point equivalent class of ODEs among a class of contactly equivalent ODEs, see \cite{K1}.

\section{A complex structure on $\V[k]$}\label{secLinear}

As before, let  $\V[k]$ be the $(k+1)$-dimensional space of homogeneous polynomials of order $k$ in two variables $(x,y)$. Our goal in this section is to equip $\V[k]$ with a complex structure (provided that $k+1=2l$ is even). This can be easily achieved by imposing the condition that the complex polynomials $z^k,z^{k-1}\bar z,\ldots,z^{k-l+1}\bar z^{l-1}$, where $z=x+\i y$, span the $\i$-eigenspace of the structure.  However, we shall start with a more explicit construction that will justify the construction from the point of view of the $CO(2)$ action. For this recall that, we consider the right action $V\cdot g$, where $g\in GL(2)$ and $V\in\V[k]$, defined by \eqref{eqaction}. Then
$$
CO(2)=\left\{ \left(\begin{array}{cc} u & -v\\ v & u\end{array}\right)\ |\ u^2+v^2\neq 0 \right\}
$$
is an Abelian subgroup of $GL(2)$. We shall consider the action of $CO(2)$ on $\V[k]$. In order to have explicit formulae we define the following polynomials
\begin{equation}\label{eqW}
\begin{aligned}
&W_i^{even}(x,y)=\sum_{s=0}^{\lfloor \frac{i}{2}\rfloor}(-1)^s\binom{i}{2s}x^{i-2s}y^{2s},\\
&W_i^{odd}(x,y)=\sum_{s=0}^{\lfloor \frac{i-1}{2}\rfloor}(-1)^s\binom{i}{2s+1}x^{i-2s-1}y^{2s+1}.
\end{aligned}
\end{equation}
Then, we have
\begin{lemma}\label{lem1}
For any $g\in CO(2)$
\begin{equation}\label{eq1}
\left(W_i^{even}\cdot g,  W_i^{odd}\cdot g\right)=\left(W_i^{even}, W_i^{odd}\right)(g^i),
\end{equation}
where on the right hand side a row vector $\left(W_i^{even}, W_i^{odd}\right)$ is multiplied by the matrix $g^i$ -- the $i$th power of $g$.
\end{lemma}
\begin{proof}
Let
$$
g=\left(\begin{array}{cc} u & -v \\ v & u\end{array}\right).
$$
If $i=1$ then $W_i^{even}(x,y)=x$ and $W_i^{odd}(x,y)=y$ and \eqref{eq1} holds directly from the definition, since $x\cdot g=(ux+vy)$ and $y\cdot g=(-vx+uy)$. If $i>1$ then we proceed by induction on $i$. Note that
$$
\begin{aligned}
W_{i+1}^{even}(x,y)=xW_i^{even}(x,y)-yW_i^{odd}(x,y),\\
W_{i+1}^{odd}(x,y)=yW_i^{even}(x,y)+xW_i^{odd}(x,y),
\end{aligned}
$$
i.e. $(W_{i+1}^{even}(x,y), W_{i+1}^{odd}(x,y))=(W_i^{even}(x,y), W_i^{odd}(x,y))\left(\begin{array}{cc} x & y \\ -y & x\end{array}\right)$. Thus, assuming \eqref{eq1} for $i$, having in mind that $CO(2)$ is Abelian.
\end{proof}

Denote
\begin{equation}\label{eqV}
\begin{aligned}
&V_{k,j}^{even}(x,y)=(x^2+y^2)^jW_{k-2j}^{even}(x,y),\\
&V_{k,j}^{odd}(x,y)=(x^2+y^2)^jW_{k-2j}^{odd}(x,y).
\end{aligned}
\end{equation}
Directly from Lemma \ref{lem1} we get
\begin{proposition}\label{prop0}
$\V[k]$ decomposes into the following sum of components irreducible with respect to $CO(2)$
$$
\V[k]=\bigoplus_{j=0}^{\lfloor\frac{k}{2}\rfloor}\V[k,k-2j],
$$
where
$$
\V[k,k-2j]=\spn\left\{V_{k,j}^{even}(x,y),\ V_{k,j}^{odd}(x,y)\right\},
$$
and
$$
\V[k,0]=\spn\{(x^2+y^2)^{\frac{k}{2}}\}
$$
provided that $k$ is even.
\end{proposition}
\begin{proof}
Taking into account Lemma \ref{lem1} it is sufficient to notice that $(x^2+y^2)$ is preserved up to a constant by $CO(2)$. Indeed, if $g=\left(\begin{array}{cc} u & -v \\ v & u\end{array}\right)\in CO(2)$ then $ (x^2+y^2)\cdot g=(u^2+v^2)(x^2+y^2)$.
\end{proof}

Note that $\dim \V[k,i]=2$ if $i>0$ and $\dim \V[k,0]=1$. We will say that $\V[k,k-2j]$ has weight $k-2j$.

Let $\fj$ be the standard complex structure on $\R^2$, $\fj(x,y)=(y,-x)$, i.e.
$$
\fj=\left(\begin{array}{cc} 0 & -1\\ 1 & 0\end{array}\right),
$$
as a matrix acting on $(x,y)$ on the right. Moreover, let
$$
\sqrt[i]{\fj}=\left(\begin{array}{cc} \cos(\frac{\pi}{2i}) & -\sin(\frac{\pi}{2i})\\ \sin(\frac{\pi}{2i}) & \cos(\frac{\pi}{2i})\end{array}\right)
$$
denotes the principal $i$-root of $\fj$, $i\in\N$. Then $\fj$ as well as $\sqrt[i]{\fj}$ for any $i\in \N$ belong to $CO(2)$. Thus, one can look how $\sqrt[i]{\fj}$ acts on each irreducible component $\V[k,i]$ for any fixed weight $i$. The following proposition is a consequence of formula \eqref{eq1} in Lemma \ref{lem1}. 
\begin{proposition}\label{prop1}
For any $i\in\N$ the $i$-root $\sqrt[i]{\fj}$ is a complex structure commuting with the action of $CO(2)$ that is unique up to the conjugation.
\end{proposition}

Let $k$ be odd. We introduce a complex structure $\fJ^k$ on $\V[k]$ in the following way
\begin{equation}\label{defJ}
\fJ^k|_{\V[k,i]}=(\sqrt[i]{\fj}).
\end{equation}
If $k$ is even, then the same formula defines a complex structure on $\V[k,>0]=\bigoplus_{j=1}^{\frac{k}{2}}\V[k,2j]$. By \eqref{eq1}, we have
\begin{equation}\label{eq2}
\fJ^k(V_{k,j}^{even})=V_{k,j}^{odd},\qquad \fJ^k(V_{k,j}^{odd})=-V_{k,j}^{even}.
\end{equation}
In particular
$$
X_{k,j}=V_{k,j}^{even}-\i V_{k,j}^{odd}
$$
is an eigenvector of $\fJ^k$ corresponding to $\i$. Note that \eqref{eqW} and \eqref{eqV} imply
\begin{equation}\label{eq3}
X_{k,j}(x,y)=(x^2+y^2)^j(x-\i y)^{k-2j}.
\end{equation}
Moreover, \eqref{eq3} can be rewritten as
\begin{equation}\label{eq3b}
X_{k,j}(x,y)=(x+\i y)^j(x-\i y)^{k-j},
\end{equation}
where now we can extend the range of the index $j$ to $0,\ldots,k$. Indeed, $X_{k,j}$ is a well defined complex-valued polynomial for all $j=0,\ldots,k$. Additionally, the following formula for the conjugate polynomials holds
$$
\bar X_{k,j}=X_{k,k-j}.
$$
Let $(\xi^{k,j}_{even},\xi^{k,j}_{odd}\ |\ j=0,1,\ldots,\lfloor\frac{k}{2}\rfloor)$ be the basis of the space $\V[k]^*$, or $\V[k,>0]^*$ if $k$ is even, dual to the basis $(V_{k,j}^{even},V_{k,j}^{odd}\ |\ j=0,1,\ldots,\lfloor\frac{k}{2}\rfloor)$ of $\V[k]^*$, or $\V[k,>0]^*$, respectively. Then the $(1,0)$-forms for $\fJ^k$ are of the form
\begin{equation}\label{eq4}
\xi^{k,j}=\xi^{k,j}_{even}+\i \xi^{k,j}_{odd},
\end{equation}
where $j=0,1,\ldots,\lfloor\frac{k}{2}\rfloor$.

We summarize our considerations in the following
\begin{proposition}\label{prop2}
$\fJ^k$ is a complex structure on $\V[k]$ (or $\V[k,>0]$ if $k$ is even) commuting with the action of $CO(2)$. On each irreducible component $\V[k,k-2j]$ the complex structure is defined by the $(1,0)$-form $\xi^{k,j}$.
\end{proposition}

\begin{remark}
As follows from Proposition \ref{prop1}, $\fJ^k$ is unique up to a conjugation on each $\V[k,i]$. Therefore, one can change $\fJ^k$ on each factor $\V[k,i]$ and get a different complex structure on $\V[k]$ that commutes with $CO(2)$ and is, a priori, as good as $\fJ^k$. However, as it will follow from \eqref{eqtor3} in Section \ref{secGL2} the definition \eqref{defJ} is the unique one (up to a conjugation on whole $\V[k]$) such that the almost-complex structure defined in Section \ref{secGL2} for the flat $GL(2)$-structure is integrable.
\end{remark}

\begin{examples}
Let us consider low-dimensional examples for future applications. We shall denote by $(\omega^0,\ldots,\omega^k)$ the basis of $\V[k]^*$ dual to the standard basis of monomials, i.e.
$$
\omega^i(x^{k-i}y^i)=\delta^i_j.
$$
If $k=1$, then
$$
\begin{aligned}
&V_{1,0}^{even}(x,y)=x,\qquad V_{1,0}^{odd}(x,y)=y,
\end{aligned}
$$
and one gets that (1,0)-forms are proportional to
$$
\xi^{1,0}=\omega^0+\i\omega^1.
$$
If $k=3$, then
$$
\begin{aligned}
V_{3,0}^{even}(x,y)=x^3-3xy^2,&\qquad V_{3,0}^{odd}(x,y)=3x^2y-y^3,\\
V_{3,1}^{even}(x,y)=x^3+xy^2,&\qquad V_{3,1}^{odd}(x,y)=x^2y+y^3,
\end{aligned}
$$
and the set of (1,0)-forms on $\V$ is spanned by
$$
\xi^{3,1}=\frac{1}{4}(3\omega^0+\omega^2+\i(\omega^1+3\omega^3)),\qquad \xi^{3,0}=\frac{1}{4}(\omega^0-\omega^2+\i(\omega^1-\omega^3)).
$$
If $k=5$, then
$$
\begin{aligned}
V_{5,0}^{even}(x,y)=x^5-10x^3y^2+5xy^4,&\qquad V_{5,0}^{odd}(x,y)=5x^4y-10x^2y^3+y^5,\\
V_{5,1}^{even}(x,y)=x^5-2x^3y^2-3xy^4,&\qquad V_{5,1}^{odd}(x,y)=3x^4y+2x^2y^3-y^5,\\
V_{5,2}^{even}(x,y)=x^5+2x^3y^2+xy^4,&\qquad V_{5,2}^{odd}(x,y)=x^4y+2x^2y^3+y^5,
\end{aligned}
$$ 
and the set of (1,0)-forms on $\V$ is spanned by
$$
\begin{aligned}
&\xi^{5,2}=\frac{1}{76}(10\omega^0+7\omega^2+12\omega^4+\i(12\omega^1+7\omega^3+10\omega^5)),\\
&\xi^{5,1}=\frac{1}{76}(5\omega^0-6\omega^2-13\omega^4+\i(13\omega^1+6\omega^3-5\omega^5)),\\
&\xi^{5,0}=\frac{1}{76}(\omega^0-5\omega^2+5\omega^4+\i(5\omega^1-5\omega^3+\omega^5)).
\end{aligned}
$$
Finally, if $k=2$, then
$$
V_{2,0}^{even}(x,y)=x^2-y^2,\qquad V_{2,0}^{odd}=2xy,
$$
and the set of (1,0)-forms on $\V[2,2]$ is spanned by
$$
\xi^{2,0}=\frac{1}{2}(\omega^0-\omega^2+\i\omega^1).
$$
\end{examples}

\section{Complex structures for $GL(2)$-structures in even dimension}\label{secGL2}
In this section we utilize the algebraic construction of Section \ref{secLinear} and define an almost-complex structure for a $GL(2)$-structure equipped with an adapted connection. Here an adapted connection is understood as a linear connection preserving the adapted co-frame bundle $F$ of Section \ref{secPre}. Thus, it is a $GL(2)$-connection in a sense of $G$-structures. In the next section we shall define a canonical connection for any $GL(2)$-structure in dimension $\geq 4$. Applying the results of the present section to the canonical connection one gets a canonical almost-complex structure constructed out of a $GL(2)$-structure.

Let $M$ be a $(k+1)$-dimensional manifold with a $GL(2)$-structure $TM\simeq S^k(E)$. Let $\pi\colon B(E)\to M$ be the corresponding reduction of the frame bundle. We assume that $k$ is odd. Let $\omega=(\omega^0,\ldots,\omega^k)^t$ be the canonical soldering form on $B(E)$ taking values in $\V[k]$ with the standard basis $(x^k, x^{k-1}y,\ldots,xy^{k-1},y^k)$. Further, let $\xi=(\xi^{k,0},\ldots,\xi^{k,\lfloor\frac{k}{2}\rfloor})^t$ be the complex-valued 1-forms on $B(E)$ composed from $\omega^i$'s using the formula \eqref{eq4}, where we treat $(\xi^{k,j}_{even},\xi^{k,j}_{odd}\ |\ j=0,1,\ldots,\lfloor\frac{k}{2}\rfloor)$ as horizontal 1-forms on $B(E)$.
\begin{theorem}\label{thm1}
Let $TM\simeq S^k(E)$ be a $GL(2)$-structure  on even-dimensional manifold $M$ and assume that a $GL(2)$-connection on $B(E)$ is defined by a 1-form $\phi=(\phi^i_j)_{i,j=1,2}$ with values in $\gl(2)$. Then there is a canonical almost-complex structure $\fJ_\phi$ on the quotient bundle
$$
B(E)/CO(2)
$$
whose $(1, 0)$-forms pullback to $B(E)$ to become linear combinations of the forms $\xi^{k,0},\ldots,\xi^{k,\lfloor\frac{k}{2}\rfloor}$ and
$$
\zeta=(\phi^1_2+\phi^2_1)+\i(\phi^2_2-\phi^1_1).
$$
\end{theorem}
\begin{proof}
We shall prove that the 1-forms $\xi^{k,j}$ and $\zeta$ are invariant up to a constant with respect to the right action of $CO(2)$. This can be checked directly using Lemma \ref{lem1} and the equivariance property of the soldering and the connection forms with respect to the $GL(2)$-action. Indeed, taking $g=\left(\begin{array}{cc} u & -v\\ v & u\end{array}\right)\in CO(2)$ and applying Lemma \ref{lem1} we get
$$
(R_g)^*\xi^{k,j}=\frac{\tilde u+\i\tilde v}{(u^2+v^2)^j}\xi^{k,j},
$$
where the coefficients $\tilde u$ and $\tilde v$ depend polynomially on $u$ and $v$ and are defined as entries of the matrix being the $(k-2j)$th power of the matrix $g^{-1}$, i.e. $(g^{-1})^{k-2j}=\left(\begin{array}{cc} \tilde u & -\tilde v\\ \tilde v & \tilde u\end{array}\right)$.  Similarly, the equivariance of $\phi$ implies
$$
(R_g)^*\zeta=\frac{1}{u^2+v^2}\left((u^2-v^2)-2\i uv\right)\zeta.
$$
\end{proof}

\begin{remark}
If $M$ is an odd-dimensional manifold then the subspace $\V[k,>0]$ of $\V[k]$ defines a corank-one distribution on $B(E)$ which, since it is $CO(2)$-invariant, descents to a corank-one distribution on $B(E)/CO(2)$. Then, analogously to Theorem \ref{thm1} one gets an almost CR-structure on $B(E)/CO(2)$ provided that a $GL(2)$-connection is given. We shall not elaborate on this in the present paper, leaving the odd-dimensional case for further research.
\end{remark}

Recall that the torsion $\Theta$ of $\phi$ is a horizontal 2-form defined by the structure equation
\begin{equation}\label{eqtor1}
d\omega=-\Phi\wedge\omega+\Theta.
\end{equation}
where
$$
\Phi=\left(\begin{array}{ccccc}
k\phi^1_1 & \phi^2_1 & 0 & \cdots &0\\
k\phi^1_2 & (k-1)\phi^1_1+\phi^2_2 & 2\phi^2_1 & \cdots& 0 \\ 
0 & (k-1)\phi^1_2 & (k-2)\phi^1_1+2\phi^2_2 & \cdots & 0 \\ 
\vdots & \vdots & \vdots & \ddots & \vdots \\
0 & 0 & 0 &\cdots & k\phi^2_2
\end{array}\right)
$$
is a 1-form with values in $\fg_k$. Similarly, the curvature $\Omega$ of $\phi$ is a horizontal 2-form defined by
\begin{equation}\label{eqcur1}
d\phi=-\phi\wedge\phi+\Omega.
\end{equation}
The structure equations can be also written in terms of $\xi$ and $\bar{\xi}$ instead of $\omega$, where $\bar\xi=(\bar\xi^{k,0},\ldots,\bar\xi^{k,\lfloor\frac{k}{2}\rfloor})^t$ is the complex conjugate of $\xi$. In particular we have the following decomposition
\begin{equation}\label{eqtor2}
d\xi=-\Lambda^0\wedge\xi-\Lambda^1\wedge\bar{\xi}+ T^{(2,0)}\xi\wedge\xi+T^{(1,1)}\xi\wedge\bar\xi+T^{(0,2)}\bar\xi\wedge\bar\xi
\end{equation}
where $\Lambda^0$ and $\Lambda^1$ are certain matrices of complex-valued 1-forms composed from the connection form $\phi$. $T^{(2,0)}$, $T^{(1,1)}$ and $T^{(0,2)}$ are torsion coefficients computed in the basis $(\xi,\bar\xi)$. Similarly we have
\begin{equation}\label{eqcur2}
d\zeta=\frac{1}{2}\i\zeta\wedge(\phi^1_2-\phi^2_1)+ C^{(2,0)}\xi\wedge\xi+C^{(1,1)}\xi\wedge\bar\xi+C^{(0,2)}\bar\xi\wedge\bar\xi,
\end{equation}
where the first term $\frac{1}{2}\i\zeta\wedge(\phi^1_2-\phi^2_1)$ is obtained by direct computations and  $C^{(2,0)}$, $C^{(1,1)}$ and $C^{(0,2)}$ are components of the curvature in the basis $(\xi,\bar\xi)$.

\begin{theorem}\label{thm2}
Let $TM\simeq S^k(E)$ be a $GL(2)$-structure on even-dimensional manifold $M$ and assume that a $GL(2)$-connection on $B(E)$ is defined by a 1-form $\phi=(\phi^i_j)_{i,j=1,2}$ with values in $\gl(2)$. Then the almost-complex structure $\fJ_\phi$ on $B(E)/CO(2)$ is integrable if and only if $T^{(0,2)}=0$ and $C^{(0,2)}=0$.
\end{theorem}
\begin{proof}
The integrability of the almost-complex structure is equivalent to the fact that
$$
d\zeta\wedge\xi^{k,0}\wedge\ldots\wedge\xi^{k,\lfloor\frac{k}{2}\rfloor}\wedge\zeta=0
$$
and
$$
d\xi^{k,j}\wedge\xi^{k,0}\wedge\ldots\wedge\xi^{k,\lfloor\frac{k}{2}\rfloor}\wedge\zeta=0
$$
for any $j=0,\ldots,\lfloor\frac{k}{2}\rfloor$.
Assume first that $T^{(0,2)}\neq 0$ or $C^{(0,2)}\neq 0$. Then there is a non-trivial component of the form $\bar\xi\wedge\bar\xi$ in \eqref{eqtor2} or in \eqref{eqcur2}. Consequently $d\xi^{k,j}\wedge\xi^{k,0}\wedge\ldots\wedge\xi^{k,\lfloor\frac{k}{2}\rfloor}\wedge\zeta\neq 0$ for some $j$ or $d\zeta\wedge\xi^{k,0}\wedge\ldots\wedge\xi^{k,\lfloor\frac{k}{2}\rfloor}\wedge\zeta\neq 0$. Therefore the almost-complex structure on $B(E)/CO(2)$ is not integrable.

Now, assume that $T^{(0,2)}=0$ and $C^{(0,2)}=0$. Then, clearly $d\zeta\wedge\xi^{k,0}\wedge\ldots\wedge\xi^{k,\lfloor\frac{k}{2}\rfloor}\wedge\zeta=0$ holds. In order to prove $d\xi^{k,j}\wedge\xi^{k,0}\wedge\ldots\wedge\xi^{k,\lfloor\frac{k}{2}\rfloor}\wedge\zeta=0$ it is sufficient to show that for any fixed $j$ the part of $d\xi^{k,j}$ involving $\Lambda^0$ and $\Lambda^1$ vanishes when multiplied by $\xi^{k,0}\wedge\ldots\wedge\xi^{k,\lfloor\frac{k}{2}\rfloor}\wedge\zeta$. Equivalently, it is sufficient to consider left representation $\Phi$ on $\V[k]^*$ and prove that any $\Phi\cdot\xi^{k,j}$, where $j=0,\ldots,\lfloor\frac{k}{2}\rfloor$, decomposes as a sum of components each involving either $\zeta$ or $\xi^{k,l}$ for some $l\in\{0,\ldots,\lfloor\frac{k}{2}\rfloor\}$. Passing to the dual representation, it is sufficient to prove that any $X_{k,j}\cdot\Phi$, where $j=\lfloor\frac{k}{2}\rfloor+1,\ldots,k$, decomposes as a sum of components each involving either $\zeta$ as a coefficient or $X_{k,l}$ for some $l\in\{\lfloor\frac{k}{2}\rfloor+1,\ldots,k\}$. Assume that $g(t)=\left(\begin{array}{cc} a(t) & b(t) \\ c(t) & d(t)\end{array}\right)$ is such that $g(t)=\exp(t\phi)$ and recall that $X_{k,l}$ is an eigenvector of $\fJ^k$ defined in \ref{defJ}. Then
$$
\begin{aligned}
&X_{k,j}\cdot g=\\
&(a(t)x+c(t)y+\i b(t)x+\i d(t)y)^j(a(t)x+c(t)y-\i b(t)x-\i d(t)y)^{k-j}.
\end{aligned}
$$
Moreover, since
$$
a'(0)=\psi^1_1,\quad d'(0)=\psi^2_2,
\quad 
b'(0)=\psi^2_1,\quad c'(0)=\psi^1_2,
$$
we get
$$
\begin{aligned}
&\left(a(t)x+c(t)y+\i b(t)x+\i d(t)y\right)'|_{t=0}=\\
&\qquad\qquad\qquad =\phi^1_1x+\phi^1_2y+\i\phi^2_1x+\i\phi^2_2y=\\
&\qquad\qquad\qquad =-\i\zeta(x+\i y)+(\phi^1_1+\phi^2_2+\i(\phi^1_2-\phi^2_1))(x-\i y).
\end{aligned}
$$
and
$$
\begin{aligned}
&\left(a(t)x+c(t)y-\i b(t)x-\i d(t)y\right)'|_{t=0}=\\
&\qquad\qquad\qquad =\phi^1_1x+\phi^1_2y-\i\phi^2_1x-\i\phi^2_2y=\\
&\qquad\qquad\qquad =\i\bar\zeta(x-\i y)+(\phi^1_1+\phi^2_2-\i(\phi^1_2-\phi^2_1))(x+\i y).
\end{aligned}
$$
Thus, differentiating $X_{k,j}\cdot g(t)$ with respect to $t$, at $t=0$, we obtain
\begin{equation}\label{eqtor3}
\begin{aligned}
X_{k,j}\cdot\Phi&=
j\left(-\i\zeta X_{k,j-1}+(\phi^1_1+\phi^2_2-\i(\phi^1_2-\phi^2_1))X_{k,j}\right)+\\
&+(k-j)\left(\i\bar\zeta X_{k,j+1}+(\phi^1_1+\phi^2_2+\i(\phi^1_2-\phi^2_1))X_{k,j}\right).
\end{aligned}
\end{equation}
Hence, $X_{k,j}\cdot\Phi$ decomposes as a sum of components as claimed.
\end{proof}

\begin{proposition}\label{prop3}
If $T^{(1,1)}=0$ and $T^{(0,2)}=0$ then $C^{(0,2)}=0$.
\end{proposition}
\begin{proof}
Assume $T^{(1,1)}=0$ and $T^{(0,2)}=0$. Then \eqref{eqtor3} implies that
$$
d\xi^{k,0}=k\alpha\wedge\xi^{k,0}-\i\zeta\wedge\xi^{k,1}\mod\  \xi\wedge\xi,
$$
where
$$
\alpha=\phi^1_1+\phi^2_2+\i(\phi^1_2-\phi^2_1).
$$
Thus
$$
\i d(\zeta\wedge\xi^{k,1})=kd(\alpha\wedge\xi^{k,0})\mod\ d\xi\wedge\xi.
$$
Expanding the both sides of the above equation we see that the only term on the left hand side involving terms of the form  $\bar\xi\wedge\bar\xi\wedge\xi$ comes from $d\zeta\wedge\xi^{k,1}$ and is equal 
$$
\i C^{(0,2)}\bar\xi\wedge\bar\xi\wedge\xi^{k,1}.
$$
On the other hand, the only term on the right hand side involving $\bar\xi\wedge\bar\xi\wedge\xi$ comes from $d\alpha\wedge\xi^{k,0}$, because, due to the assumption, there are no terms of the form $\bar\xi\wedge\xi$ or $\bar\xi\wedge\bar\xi$ in the torsion.  Therefore, the right hand side does not contain any term of the form $\bar\xi\wedge\bar\xi\wedge\xi^{k,1}$. It follows that $C^{(0,2)}=0$. 
\end{proof}

\section{Canonical connections}\label{secCon}

In this section we construct canonical connections for all $GL(2)$-structures in dimension 4 or greater. As a consequence, applying results of Section \ref{secGL2} to the canonical connections, we are able to assign a canonical almost-complex structure to any $GL(2)$-structure. 

Let $M$ be a manifold of dimension $k+1$,  not necessarily even, and let $TM\simeq S^k(E)$ be a $GL(2)$-structure on $M$. As before $\omega$ is the canonical $\V[k]$-valued soldering form and $\phi$ is a $\gl(2)$-valued connection form defined on $B(E)$. Recall that the torsion 2-form $\Theta$ is defined by equation \eqref{eqtor1}. For each $\chi\in B(E)$ we have $\Theta_\chi\in(\V[k]^*\wedge\V[k]^*)\otimes\V[k]$.

Recall that $\fg_k$ is a subalgebra of $\V[k]^*\otimes\V[k]$, isomorphic to $\gl(2)$, corresponding to the standard representation of $GL(2)$. Using $\fg_k$ we define the following subspace of $\V[k]^*\otimes\V[k]$
$$
\fg_k^\perp=\{\psi\in \V[k]^*\otimes\V[k]\ |\ \tr(\eta\circ\psi)=0\quad \forall \eta\in\fg_k\}.
$$
A matrix $\psi=(\psi^i_j)_{i,j=0,\ldots,k}$ is in $\fg_k^\perp$ if and only if
\begin{equation}\label{eqnorm}
\begin{aligned}
&\sum_{i=0}^k\psi^i_i=0,\quad\sum_{i=0}^k(k-2i)\psi^i_i=0,\\
&\sum_{i=0}^{k-1}(k-i)\psi^{i+1}_i=0,\quad\sum_{i=1}^ki\psi^{i-1}_i=0.
\end{aligned}
\end{equation}
The condition $\tr(\eta\circ\psi)=0$ is invariant with respect to the action of $GL(2)$ on $\V[k]^*\otimes\V[k]$\footnote{Both $\eta$ and $\psi$ can be considered as $(k+1)\times(k+1)$ matrices. Then $\tr(\eta\circ\psi)=\tr(\eta\psi)$. Note that this is different than the scalar product of matrices $\tr(\eta^t\psi)$}. Thus it can be used as a condition defining a connection. 

\begin{theorem}\label{thmnorm}
Let $TM\simeq S^k(E)$ be a $GL(2)$-structure on a manifold $M$ of dimension $k+1>3$. There is a unique $GL(2)$-connection $\phi=(\phi^i_j)_{i,j=1,2}$ with values in $\gl(2)$ such that $\Theta_\chi(X,.)\in\fg_k^\perp$ for any $\chi\in B(E)$ and $X\in T_\chi B(E)$.
\end{theorem}
\begin{proof}
Let $\phi=(\phi^i_j)_{i,j=1,2}$ be any $GL(2)$-connection on $M$. We shall prove that $\phi$ can be uniquely modified $\phi\mapsto \phi+\psi$ by a horizontal 1-form $\psi=(\psi^i_j)_{i,j=1,2}$  such that the normalization condition is satisfied. Recall that the torsion transforms as 
$$
\Theta(X,Y)\mapsto \Theta(X,Y)+\Psi(Y)X-\Psi(X)Y,
$$
where $\Psi$ is a 1-form with values in $\fg_k$ corresponding to $\psi$, whereas $\Psi(Y)X$ and $\Psi(X)Y$ denote multiplication of $\Psi(Y)$ and $\Psi(X)$, considered as elements of $\fg_k$, by column vectors $X$ or $Y$, respectively. Let $X_i=x^iy^{k-i}$ be an element of $\V[k]$. Denote $\psi^l_j(X_i)=\psi^l_{ji}$. We shall consider the linear mapping $Y\mapsto \Psi(Y)X_i-\Psi(X_i)Y$ whose matrix takes the following form
$$
\begin{aligned}
&\Psi X_i-\Psi(X_i)=\\
&=\left(\begin{array}{ccccc}
0 & 0 & \cdots & 0\\
\vdots & \vdots & \ddots & \vdots\\
0 & 0 & \cdots & 0\\
i\psi^2_{10} & i\psi^2_{11} & \cdots & i\psi^2_{1k}\\
(k-i)\psi^1_{10}+i\psi^2_{20} & (k-i)\psi^1_{11}+i\psi^2_{21} & \cdots & (k-i)\psi^1_{1k}+i\psi^2_{2k}\\
(k-i)\psi^1_{20} & (k-i)\psi^1_{21} & \cdots & (k-i)\psi^1_{2k}\\
0 & 0 & \cdots & 0\\
\vdots & \vdots & \ddots & \vdots\\
0 & 0 & \cdots & 0\\
\end{array}\right)\\
&\qquad-\left(\begin{array}{ccccc}
k\psi^1_{1i} & \psi^2_{1i} & 0 &\cdots &0\\
k\psi^1_{2i} & (k-1)\psi^1_{1i}+\psi^2_{2i} & 2\psi^2_{1i} & \cdots& 0 \\ 
0 & (k-1)\psi^1_{2i} & (k-2)\psi^1_{1i}+2\psi^2_{2i} & \cdots & 0 \\ 
\vdots & \vdots & \vdots & \ddots & \vdots \\
0 & 0 & 0 & \cdots & k\psi^2_{2i}
\end{array}\right).
\end{aligned}
$$
Now, it is enough to prove that the conditions $\Psi X_i-\Psi(X_i)\in\fg_k^\perp$, considered jointly for all $i=0,\ldots,k$, have unique solution $\psi^l_{ji}=0$. Applying \eqref{eqnorm} to the matrices $\Psi X_i-\Psi(X_i)$, we get a system of $4(k+1)$ linear equations for unknown $\psi^l_{ji}$. Note that the equations can be grouped into $k+1$ subsystems, of 4 equations each, for variables $\psi^1_{1i},\psi^2_{2i}, \psi^2_{1i-1}, \psi^1_{2i+1}$. Indeed, we have the following
$$
\begin{aligned}
&\left(\binom{k+1}{2}-(k-i)\right)\psi^1_{1i}+\left(\binom{k+1}{2}-i\right)\psi^2_{2i}\\
&\qquad -i\psi^2_{1i-1} -(k-i)\psi^1_{2i+1}=0\\
&\left(\binom{k+2}{6}-(k-2i)(k-i)\right)\psi^1_{1i} -\left(\binom{k+2}{6}+(k-2i)i\right)\psi^2_{2i} \\ &\qquad-(k-2(i-1))i\psi^2_{1i+1}-(k-2(i+1))(k-i)\psi^1_{2i+1}=0
\end{aligned}
$$
$$
\begin{aligned}
&-(i+1)(k-(i+1))\psi^1_{1i} -(i+1)^2\psi^2_{2i} \\
&\qquad-i(i+1)\psi^2_{1i+1}+\left(\binom{k+2}{3}-(i+2)(k-(i+1))\right)\psi^1_{2i+1}=0\\
&-(k-(i-1))^2\psi^1_{1i} -(k-(i-1))(i-1)\psi^2_{2i} \\
&\qquad+\left(\binom{k+2}{3}-(k-(i-2))(i-1)\right)\psi^2_{1i+1}-(k-i)(k-(i-1))\psi^1_{2i+1}=0
\end{aligned}
$$
where $i=0,\ldots,k$. Surprisingly, the determinants of the above systems are independent of $i$ and equal
$$
\frac{1}{216}(k-2)(k-1)^2k^3(k+2)(k+3)(k+4)(k^2+k+6).
$$
Thus, they are non-degenerate, provided $k>2$. 
\end{proof}

The unique connection $\phi$ satisfying the condition of Theorem \ref{thmnorm} as well as the corresponding almost complex structure $\fJ_\phi$ will be referred to as the canonical connection, and the canonical almost-complex structure, respectively. 

\begin{examples}
If $k=2$ then the torsion of a $GL(2)$-connection has 9 components, whereas the connection itself depends on 12 arbitrary functions. Thus, there is a 3-parameter family of torsion-free connections associated to a $GL(2)$-structure which clearly cannot be further normalized in terms of the torsion. However, the geometry of 3-dimensional $GL(2)$-structures is the geometry of conformal metrics and we refer to \cite{K} for a construction of a canonical Cartan connection in this case.

If $k=3$ then the normalization condition of Theorem \ref{thmnorm} coincides with the condition of Bryant \cite{B} (see also \cite{N} for explicit formulae). Indeed, the Bryant connection is the unique $GL(2)$-connection with $\V[3]$-valued torsion of the form $\Theta=\langle\tau,\langle\omega,\omega\rangle_1\rangle_4$ in the notation of \cite{B}, where $\tau$ is a homogeneous polynomial of order 7 of variables $(x,y)$. This can be expanded to the following formula
$$
\begin{aligned}
&\Theta=x^3(\tau_0 2520\omega_2\wedge\omega_3-\tau_1 720\omega_1\wedge\omega_3+\tau_2 120(3\omega_0\wedge\omega_3+\omega_1\wedge\omega_2)\\
&\quad-\tau_3 144\omega_0\wedge\omega_2 +\tau_4 72\omega_0\wedge\omega_1)+\\
&\quad x^2y(\tau_1 1080\omega_2\wedge\omega_3-\tau_2 720\omega_1\wedge\omega_3+\tau_3 216(3\omega_0\wedge\omega_3+\omega_1\wedge\omega_2)\\
&\quad-\tau_4 432\omega_0\wedge\omega_2 +\tau_5 360\omega_0\wedge\omega_1)+\\
&\quad xy^2(\tau_2 360\omega_2\wedge\omega_3-\tau_3 432\omega_1\wedge\omega_3+\tau_4 216(3\omega_0\wedge\omega_3+\omega_1\wedge\omega_2)\\
&\quad-\tau_5 720\omega_0\wedge\omega_2 +\tau_6 1080\omega_0\wedge\omega_1)+\\
&\quad y^3(\tau_3 72\omega_2\wedge\omega_3-\tau_4 144\omega_1\wedge\omega_3+\tau_5 120(3\omega_0\wedge\omega_3+\omega_1\wedge\omega_2)\\
&\quad-\tau_6 720\omega_0\wedge\omega_2 +\tau_7 2520\omega_0\wedge\omega_1),
\end{aligned}
$$
where $(\tau_0,\ldots,\tau_7)$ are certain functions. A matrix $(\psi^i_j)_{i,j=0,\ldots,3}$ belongs to $\fg_k^\perp$ if and only if \eqref{eqnorm} is satisfied with $k=3$, i.e.
$$
\begin{aligned}
&\psi^0_0+\psi^1_1+\psi^2_2+\psi^3_3=3\psi^0_0+\psi^1_1-\psi^2_2-3\psi^3_3=0\\
&\psi^0_1+2\psi^1_2+3\psi^2_3=3\psi^1_0+2\psi^2_1+\psi^3_2=0.
\end{aligned}
$$
It can be directly verified that in the Bryant's case $\Theta_\chi(X,.)\in\fg_k^\perp$ for any vector $X\in T_\chi B(E)$. Note that one of the normalization conditions is $\tr\Theta_\chi(X,.)=0$. This, in dimension 4, can be equivalently expressed as $T^{(2,0)}=0$.
\end{examples}

\section{Holomorphic sections and twistor theory}\label{secHol}
In this section we shall exploit our previous constructions in the case of half-integrable $GL(2)$-structures. In particular we shall concentrate on 4-dimensional manifolds. In this case the half-integrability is equivalent to 3-integrability as explained in Section \ref{secPre} (Proposition \ref{propInt}). We strat with a different viewpoint on the quotient bundle $B(E)/CO(2)$.

Let us assume that $M$ is a manifold of even dimension $k+1$ with a $GL(2)$-structure $TM\simeq S^k(E)$. Let $E^\C$ denotes the complexification of $E$ and let $P(E^\C)$ be the complex projectivization of $E^\C$. Thus, for each $p\in M$ we have $P(E^\C_p)\simeq\C P^1$. Recall that there is a natural embedding of $P(E)$ into $P(E^\C)$. Therefore we shall consider $P(E)$ as a subset of $P(E^\C)$. Note that $P(E^\C_p)\setminus P(E_p)$ is a sum of two disjoint discs.
\begin{proposition}\label{propIdentif}
There is a canonical identification
$$
f_E\colon B(E)/CO(2)\to P(E^\C) \setminus P(E).
$$
\end{proposition}
\begin{proof}
Fix a point $\chi\in B(E)_p$ . It is a frame in $T_pM$ of the form
$$
 \chi=(x^k,x^{k-1}y,\ldots,xy^{k-1},x^k)
$$
uniquely defined by a basis $(x,y)$ of $E_p$. Then
$$
z=x+\i y
$$
is a point in $E^\C$. The action of $g= \left(\begin{array}{cc} u & -v\\ v & u\end{array}\right)\in CO(2)$ on $B(E)$ yields new point in $\tilde z\in E^\C$, where $\tilde z = (u+\i v)z$. Thus, we have shown that there is a natural map from $B(E)/CO(2)$ to $P(E^\C)$, $f_E([\chi])=[z]$. Since $(x,y)$ is a basis of $E_p$ we get that the map takes values in $P(E^\C) \setminus P(E)$, because $P(E^\C) \setminus P(E)$ consist exactly of classes $[x+\i y]$ such that  $x$ and $y$ are not proportional (recall that $x$ and $y$ are vectors in $T_pM$). The reasoning can be inverted and it shows that any point $[x+\i y]\in P(E^\C)\setminus P(E)$ defines a basis in $E$ up to the action of $CO(2)$. Thus $f_E$ is invertible.
\end{proof}

Now, a point in $P(E^\C_p) \setminus P(E_p)$ can be identified with a complex structure in $T_pM$ as follows.

A complex structure on $T_pM$ is uniquely defined by a half-dimensional subspace in $T^\C_pM$ being, by definition, the $\i$-eigenspace. If a $GL(2)$-structure $TM\simeq S^k(E)$ is present then $T^\C_pM\simeq S^k(E^\C_p)$. Fixing a basis $(x,y)$ in $E_p$ and denoting $z=x+\i y$ and $\bar z=x-\i y$ one can consider a subspace spanned by
$$
z^k, z^{k-1}\bar z,\ldots, z^{\frac{k+1}{2}}\bar z^{\frac{k-1}{2}}
$$
and declare that this subspace defines a complex structure $\fJ_z$ on $TM$. The structure depends on the class $[z]\in P(E^\C_p)\setminus P(E_p)$ only. Hence, points in $P(E^\C_p)\setminus P(E_p)$ and consequently, due to Proposition \ref{propIdentif}, points in $B(E)_p/CO(2)$ can be considered as complex structures on $T_pM$. 

\begin{proposition}\label{propComplIdent}
For any basis $(x,y)$ of $E_p$ the complex structure $\fJ_z$ on $T_pM$, where $z=x+\i y$, coincides with $\fJ^k$ on $\V[k]$ under the identification $T_pM\simeq\V[k]$ defined by the basis $(x,y)$.
\end{proposition}
\begin{proof}
A direct consequence of formula \eqref{eq3b}.
\end{proof}

\begin{example}
If $k=1$ then $M$ is a two-dimensional manifold and $E=TM$. Then $P(T^\C M)\setminus P(TM)$ is the bundle of complex structures on $TM$. Indeed, any complex direction $[z]\in P(T^\C_p M)$ defines a complex structure in $T_pM$. The complex structures on $T_pM$ are in a one-to-one correspondence with conformal metrics on $T_pM$. Thus $P(T^\C M)\setminus P(TM)\simeq B(E)/CO(2)$, where $B(E)$ is the full frame bundle in this case (c.f. \cite{M2} for applications of this constructions).
\end{example}

Now, the almost-complex structure $\fJ_\phi$ in Theorem \ref{thm1} can be equivalently understood as follows. The connection $\phi$ splits the tangent bundle $T(B(E)/CO(2))$ into the horizontal part $H_\phi$ and the vertical part $V_\phi$. The horizontal part is canonically identified with $\V[k]$ via the soldering form $\omega$, while the vertical part is identified with $\V[2,2]$ (in the notation of Section \ref{secLinear}). On each part the complex structures $\fJ^k$ and $\fJ^2$, respectively, are defined as in Section \ref{secLinear}. Additionally, $\fJ^k$ on $H_\phi$ at point $\chi\in B(E)_p/CO(2)$ can be considered as a lift of $\fJ_z$ from $T_pM$ to $H_\phi$ where $[z]=f_E(\chi)$ is defined in Proposition \ref{propIdentif}.

\begin{theorem}\label{thm3}
Let $TM\simeq S^k(E)$ be a $GL(2)$-structure on even-dimensional manifold $M$  and assume that the almost-complex structure $\fJ_\phi$ defined by a $GL(2)$-connection $\phi$ is integrable. Then any holomorphic  section of the quotient bundle $B(E)/CO(2)$ is a complex structure on $M$.

\end{theorem}
\begin{proof}
Let $\sigma\colon M\to B(E)/CO(2)$ be a section of $B(E)/CO(2)$. For any $p\in M$ denote $[z_p]=f_E(\sigma(p))\in P(E^\C_p)\setminus P(E_p)$. Then $p\mapsto \fJ_{z_p}$ is an almost complex structure on $M$, denoted $\fJ_\phi^\sigma$. Lemma 3.1 in \cite{M1} implies that $\sigma$ is holomorphic if and only if
\begin{equation}\label{eqholom}
\sigma^*(\zeta\wedge\xi^0\wedge\ldots\wedge\xi^{\lfloor\frac{k}{2}\rfloor})=0,
\end{equation}
where $\zeta$ and $\xi^i$, $i=0,\ldots,\lfloor\frac{k}{2}\rfloor$, are $(1,0)$-forms for $\fJ_\phi$. Since $\fJ_\phi$ is integrable $d\xi^i\wedge\zeta\wedge\xi^0\wedge\ldots\wedge\xi^{\lfloor\frac{k}{2}\rfloor}=0$ holds for any $i$. Hence, it follows from \eqref{eqholom} that $d(\sigma^*\xi^i)\wedge\sigma^*\xi^0\wedge\ldots\wedge\sigma^*\xi^{\lfloor\frac{k}{2}\rfloor}=0$. But $\sigma^*\xi^i$, where $i=0,\ldots,\lfloor\frac{k}{2}\rfloor$, are $(1,0)$-forms for $\fJ_\phi^\sigma$. Thus $\fJ_\phi^\sigma$ is integrable.

\end{proof}

Further, we get the following integrability result.
\begin{theorem}\label{thm4}
Let $TM\simeq S^k(E)$ be a $GL(2)$-structure on even-dimensional manifold $M$  and assume that the almost-complex structure $\fJ_\phi$ defined by a $GL(2)$-connection $\phi$ is integrable. Then, the $GL(2)$-structure is half-integrable.
\end{theorem}
\begin{proof}
Recall that $C_p(E)$ denotes the cone of null directions of $S^k(E)$. We introduce a cone of complex null directions in $T^\C_pM\simeq S^k(E^\C)$ by the analogous formula
$$
C_p(E^\C)=\left\{\underbrace{e\odot\cdots\odot e}_k\ |\ e\in E^\C_p \right\}.
$$
Let $(x,y)$ be a basis in $E_p$. Then $(z,\bar z)$ is a basis of $E^\C_p$, where $z=x+\i y$ and $\bar z=x-\i y$. Any point in $C_p(E^\C)$ can be identified with a complex polynomial
$$
V(z,\bar z)=(az+b\bar z)^k
$$
where $a,b\in\C$ are constants. Now, as in the real case, we can define complex $\alpha_i$-planes as tangent planes to osculating cones of $C_p(E^\C)$ consisting of complex polynomials with multiple roots. In particular, the $\alpha_i$-plane at $z^k\in C_p(E^\C)$ is spanned by polynomials
$$
z^k, z^{k-1}\bar z,\ldots, z^{k-l+1}\bar z^{l-1}.
$$
Let $l=\frac{k+1}{2}$. We get the following
\begin{lemma}
The space of $(0,1)$-vectors of the complex structure $\fJ_z$ on $T_pM$ is the complex $\alpha_l$-plane at $z^k\in C_p(E^\C)$.
\end{lemma}
Let $\chi\in B(E)/CO(2)$ be fixed and let $\sigma$ be a holomorphic section of $B(E)/CO(2)$ passing through $\chi$. Then, by Theorem \ref{thm3},  $p\mapsto\fJ_{f_E(\sigma(p))}$ is a complex structure on $M$ and consequently the complex distribution
$$
p\mapsto\spn\{z_p^k, z_p^{k-1}\bar z_p,\ldots, z_p^{k-l+1}\bar z_p^{l-1}\}
$$
is involutive, where $z_p=f_E(\sigma(p))$. Thus, any complex $\alpha_l$-plane extends to an involutive distribution of complex $\alpha_l$-planes. By continuity, any real $\alpha_l$-plane extends to an integrable distribution that is tangent to an $\alpha_l$-submanifold of $M$.
\end{proof}

In dimension 4 the result can be strengthen.
\begin{theorem}
A 4-dimensional $GL(2)$-structure is half-integrable if and only if the corresponding canonical almost-complex structure is integrable.
\end{theorem}
\begin{proof}
By \cite{B} a 4-dimensional structure is integrable if and only if it is torsion-free. By Proposition \ref{prop3} and Theorem \ref{thm2}, if a $GL(2)$-structure is torsion-free then $C^{(0,2)}=0$ and consequently the corresponding almost-complex structure is integrable.
\end{proof}

\end{document}